\documentclass[12pt]{amsart}

\usepackage[pdfauthor   = {Mohamed\ Barakat\ and\ Markus\ Lange-Hegermann},
            pdftitle    = {On\ monads\ of\ exact\ reflective\ localizations\ of\ Abelian\ categories},
            pdfsubject  = {},
            pdfkeywords = {Constructive\ homological\ algebra; Gabriel\ localization; Gabriel\ monads; coherent\ sheaves;},
            bookmarks=true,
            bookmarksopen=true,
            pagebackref=true,
            hyperindex=true,
            colorlinks=true,
            linkcolor=blue,
            citecolor=blue,
            filecolor=blue,
            urlcolor=blue,
            ]{hyperref}

\usepackage[utf8,utf8x]{inputenc}
\usepackage[english]{babel}
\usepackage[T1]{fontenc}
\usepackage{geometry}                
\usepackage{times}
\usepackage[novbox]{pdfsync}
\usepackage{mathrsfs}
\usepackage{latexsym}
\usepackage{amssymb}
\usepackage{amsthm}
\usepackage{epsfig}
\usepackage{colortbl}
\usepackage[all]{xy}
\usepackage{fancyvrb}
\usepackage{graphicx}
\usepackage[dvipsnames]{xcolor}
\usepackage{accents} 
\usepackage{enumerate}

\usepackage{wrapfig}
\usepackage{tikz}
\usetikzlibrary{shapes,arrows,matrix,backgrounds,positioning,plotmarks,calc,patterns,matrix,
decorations.shapes,
decorations.fractals,
decorations.markings,
decorations.pathreplacing,
decorations.pathmorphing,
decorations.text
}
\usepackage[colorinlistoftodos,shadow]{todonotes}

\usepackage{multirow}
\usepackage{mdwlist}
\usepackage{stmaryrd}
\usepackage{mathdots} 

\usepackage[toc,page]{appendix}
\usepackage{float}

\newtheoremstyle{mytheoremstyle} 
    {5pt}                    
    {5pt}                    
    {\itshape}                   
    {\parindent}                           
    {\bf}                   
    {.}                          
    {.5em}                       
    {}  

\theoremstyle{mytheoremstyle}

\newtheorem{theorem}{Theorem}[section]

\newtheorem{lemm}[theorem]{Lemma}
\newtheorem{prop}[theorem]{Proposition}
\newtheorem{coro}[theorem]{Corollary}

\newtheoremstyle{mytdefintionstyle} 
    {5pt}                    
    {5pt}                    
    {\rm}                   
    {\parindent}                           
    {\bf}                   
    {.}                          
    {.5em}                       
    {}  

\theoremstyle{remark}
\newtheorem{rmrk}[theorem]{Remark}

\theoremstyle{mytdefintionstyle}
\newtheorem{defn}[theorem]{Definition}

\newtheoremstyle{exmp_contd}
    {5pt}                    
    {5pt}                    
    {\rm}                   
    {\parindent}                           
    {\bf}                   
    {.}                          
    {.5em}                       
    {\thmname{#1}\ \thmnumber{ #2}\thmnote{#3}\ (continued)}  
\theoremstyle{exmp_contd}

\newcommand\nameft\textrm

\newcommand{\QQ}{{\mathscr{Q}}}
\renewcommand{\SS}{{\mathscr{S}}}
\newcommand{\FF}{{\mathscr{F}}}
\newcommand{\GG}{{\mathscr{G}}}
\newcommand{\HH}{{\mathscr{H}}}
\newcommand{\WW}{{\mathscr{W}}}

\DeclareMathOperator{\Ext}{Ext}

\DeclareMathOperator{\Hom}{Hom}

\DeclareMathOperator{\Sat}{Sat}

\newcommand{\Coh}{\mathfrak{Coh}\,}

\newcommand\A{\mathcal{A}}
\newcommand\C{\mathcal{C}}

\newcommand\B{\mathcal{B}}

\newcommand\F{\mathcal{F}}

\newcommand{\Z}{\mathbb{Z}}

\renewcommand\phi{\varphi}

\DeclareMathOperator\Id{Id}

\DeclareMathOperator\cores{co-res}

\definecolor{darkgray}{rgb}{0.3,0.3,0.3}

\pgfarrowsdeclarecombine[\pgflinewidth]
  {doublestealth}{doublestealth}{stealth'}{stealth'}{stealth'}{stealth'}

\definecolor{darkgreen}{rgb}{0.008,0.617,0.067}
\definecolor{brown}{rgb}{0.6,0.4,0.2}

\newif\ifjournalversion

\author{Mohamed Barakat}
\address{Department of mathematics, University of Kaiserslautern, 67653 Kaiserslautern, Germany}
\email{\href{mailto:Mohamed Barakat <barakat@mathematik.uni-kl.de>}{barakat@mathematik.uni-kl.de}}

\author{Markus Lange-Hegermann}
\address{Lehrstuhl B f\"ur Mathematik, RWTH Aachen University, 52062 Aachen, Germany}
\email{\href{mailto:Markus Lange-Hegermann <markus.lange.hegermann@rwth-aachen.de>}{markus.lange.hegermann@rwth-aachen.de}}

\begin{document}

\title[Monads of reflective localizations of \nameft{Abel}ian categories]{On monads of exact reflective localizations of \nameft{Abel}ian categories}

\begin{abstract}
  In this paper we define \nameft{Gabriel} monads as the idempotent monads associated to exact reflective localizations in \nameft{Abel}ian categories and characterize them by a simple set of properties.
  The coimage of a \nameft{Gabriel} monad is a \nameft{Serre} quotient category.
  The \nameft{Gabriel} monad induces an equivalence between its coimage and its image, the localizing subcategory of local objects.
\end{abstract}


\maketitle

\section{Introduction}

\nameft{Abel}ian categories were introduced in \nameft{Grothendieck}'s T\^ohoku paper \cite{Tohoku}, and since then became a central notion in homological algebra.
In our attempt to establish a constructive setup for homological algebra we introduced in \cite[Chap.~2]{BL} the notion of a computable \nameft{Abel}ian category, i.e.,  an \nameft{Abel}ian category in which all existential quantifiers occurring in the defining axioms can be turned into algorithms.
Along these lines we treated in loc.~cit.\  the \nameft{Abel}ian categories of finitely presented modules over so-called computable rings and their localization at certain maximal ideals.

Our next goal is to treat the \nameft{Abel}ian category $\Coh X$ of coherent sheaves on a projective scheme $X$ along the same lines.
This category is, by \nameft{Serre}'s seminal paper \cite{FAC}, equivalent to a \nameft{Serre} quotient $\A/\C$ of an \nameft{Abel}ian category $\A$ of graded modules over the graded ring $S=k[x_0,\ldots,x_n]/I$, where $I=I(X)$ is the homogeneous ideal defining $X$ and $\C$ is the thick subcategory of graded modules with zero sheafification.
In the context of this paper we require the thick subcategory $\C \subset \A$ to be localizing.
Indeed, there are several ways to model $\Coh X$ as a \nameft{Serre} quotient $\A/\C$ where $\A$ is some \nameft{Abel}ian category of graded $S$-modules, but not in all models is the thick subcategory $\C \subset \A$ localizing.
\nameft{Serre} defined in loc.~cit.\  $\A$ as the category of quasi finitely generated graded $S$-modules, i.e., of graded modules $M$ such that the truncated submodule $M_{\geq d}$ is finitely generated for some $d \in \Z$ large enough.
Here $\C \subset \A$ is localizing but this model is not constructive.
Redefining $\A,\C$ to be the respective full subcategories of finitely generated graded $S$-modules indeed yields a constructive model $\A/\C \simeq \Coh X$, but now $\C \subset \A$ is no longer localizing.
Luckily, for each $d_0 \in \Z$ the respective full subcategories $\A,\C$ of graded $S$-modules truncated at $d_0$ yield a constructive model $\A/\C \simeq \Coh X$ in which $\C \subset \A$ is again localizing (for more details cf.~\cite[§4.1]{BL_SerreQuotients}).
In all models a coherent sheaf $\F$ is given by a class of graded modules isomorphic in high degrees.
In the last model the truncated (and hence finitely generated) module of twisted global sections $\bigoplus_{d \geq d_0} H^0(\F(d))$ is in the following sense a distinguished representative within this class; it is a so-called saturated object with respect to $\C$.

The appropriate categorical setup for a \nameft{Serre} quotient $\A / \C$ of an \nameft{Abel}ian category $\A$ modulo a thick subcategory $\C \subset \A$ was introduced by \nameft{Grothendieck} \cite[Chap.~1.11]{Tohoku} and then more elaborately in \nameft{Gabriel}'s thesis \cite{Gab_thesis}.
Later \nameft{Gabriel} and \nameft{Zisman} developed in \cite{GabZis} a localization theory of categories in which a \nameft{Serre} quotient $\A/\C$ is an outcome of a special localization $\A \to \A[\Sigma^{-1}] \simeq \A / \C$.
Their theory is also general enough to enclose \nameft{Verdier}'s localization of triangulated categories, which he used in his 1967 thesis (cf.~\cite{Verdier}) to define derived categories.
Thanks to \nameft{Simpson}'s work \cite{SimpsonGZComputer} the \nameft{Gabriel-Zisman} localization is now completely formalized in the proof assistant \nameft{Coq} \cite{coq:manual}.
In many applications, as assumed in \cite{GabZis}, the localization $\A[\Sigma^{-1}]$ is equivalent to a full subcategory of $\A$, the subcategory of all $\Sigma$-local objects of $\A$.
This favorable situation (which in our context means that $\C$ is a localizing subcategory of $\A$) is characterized by the existence of an idempotent monad associated to the localization.
For a further overview on localizations we refer to the \texttt{arXiv} version of \cite{ThomasArrow}.

In our application to $\Coh X$ we are in the setup of \nameft{Serre} quotient categories $\A / \C$, which are the outcome of an exact localization having an associated idempotent monad.
We call this monad the \nameft{Gabriel} monad (cf.~Definition~\ref{defn:Gabriel_monad}).
\nameft{Gabriel} monads satisfy a set of properties which we use as a simple set of axioms to define what we call a $\C$-saturating monad.

The goal of this paper is to characterize \nameft{Gabriel} monads conversely as $\C$-saturating monads (Theorem~\ref{thm:equivalent_section}).
Such a characterization enables us in \cite{BL_Sheaves} to show that several known algorithmically computable functors in the context of coherent sheaves on a projective scheme\footnote{This technique applies to other classes of varieties admitting a finitely generated \nameft{Cox} ring $S$.
} are $\C$-saturating\footnote{In this context the $\Sigma$-local objects are \nameft{Gabriel}'s $\C$-saturated objects.}.
This yields a constructive, unified, and simple proof that those functors are equivalent to the functor $M \mapsto \bigoplus_{d \geq d_0} H^0(\widetilde{M}(d))$, and hence compute the truncated module of twisted global sections.
Among those are the functors computing the graded ideal transform and the graded module given by the $0$-th strand of the \nameft{Tate} resolution.

The proof there relies on checking the defining set of axioms of a $\C$-saturating monad, which turns out to be a relatively easy task.
In particular, the proof does not rely on the (full) \nameft{BGG} correspondence \cite{BGG} of triangulated categories, the \nameft{Serre-Grothendieck} correspondence \cite[20.3.15]{BrSh}, or the local duality, as used in \cite{EFS}.

A stronger computability notion is that of the $\Ext$-computability.
Furthermore, in \cite{BL_ExtComputability} we use the \nameft{Gabriel} monad of a localizing \nameft{Serre} quotient $\A/\C$ to show that the so-called $\Ext$-computability of $\A/\C$ follows from that of $\A$.
In particular, the \nameft{Abel}ian category $\Coh X$ is $\Ext$-computable (cf.~\cite{BL_Sheaves}).

\section{Preliminaries} \label{sect:Preliminaries}

In this short section we collect some standard categorical preliminaries, which should make this paper self-contained for the readers interested in our applications to coherent sheaves \cite{BL_Sheaves}.
For details we refer to \cite{Borceux_Handbook1,Borceux_Handbook2} or the active $n$\textsf{Lab} wiki \cite{nLab}.

\subsection{Adjoint functors and monads}

\begin{defn}
Two functors $\FF: \A \to \B$ and $\GG: \B \to \A$ are called a \textbf{pair of adjoint functors} and denoted by $\FF \dashv (\GG: \B \to \A)$ if there exists a binatural isomorphism of $\Hom$ functors
\[
  \Hom_\B(\FF(-),-) \simeq \Hom_\A(-,\GG(-)) \mbox{.}
\]
$\FF$ is called a \textbf{left adjoint} of $\GG$ and $\GG$ is a \textbf{right adjoint} of $\FF$.
\end{defn}

\begin{rmrk} \label{rmrk:right_adjoints}
  Two right adjoints (resp. two left adjoints) are naturally isomorphic.
\end{rmrk}

\begin{prop}
  An adjunction $\FF \dashv (\GG: \B \to \A)$ is characterized by the existence of two natural transformations
  \[
    \delta: \FF \circ \GG \to \Id_\B \quad\mbox{and}\quad \eta:\Id_\A \to \GG \circ \FF\mbox{,}
  \]
  called the \textbf{counit}\footnote{$\delta_b$ corresponds to $1_{\GG b}$ under the natural isomorphism
$\Hom_\B((\FF \circ \GG)b,b) \cong \Hom_\A(\GG b,\GG b)$.} and \textbf{unit\footnote{$\eta_a$ corresponds to $1_{\FF a}$ under the natural isomorphism
$\Hom_\A(a,(\GG\circ \FF) a) \cong \Hom_\B(\FF a,\FF a)$.} of the adjunction}, respectively, such that the two \textbf{zig-zag} (or \textbf{unit-counit}) \textbf{identities} hold, i.e., the compositions of the natural transformations
  \begin{equation} \tag{zz} \label{zigzag}
      \FF \xrightarrow{\FF\eta} \FF\circ \GG \circ \FF \xrightarrow{\delta \FF} \FF \quad\mbox{and}\quad \GG \xrightarrow{\eta \GG} \GG\circ \FF \circ  \GG \xrightarrow{\GG \delta} \GG
  \end{equation}
  must be the identity on functors.
\end{prop}

\begin{defn}
  A \textbf{monad} on a category $\A$ is an endofunctor $\WW:\A \to \A$ together with two natural transformations
  \[
    \eta: \Id_\A \to \WW \quad\mbox{and}\quad \mu: \WW^2 \to \WW \mbox{,}
  \]
  called the \textbf{unit} and \textbf{multiplication of the monad}, respectively, such that following two \textbf{coherence conditions} hold:
  \[
    \mu \circ \WW \mu = \mu \circ \mu \WW \quad\mbox{and}\quad \mu \circ \WW\eta= \mu \circ \eta \WW = \Id_\WW \mbox{.}
  \]
\end{defn}

\begin{rmrk}
  Let $\FF \dashv (\GG:\B \to \A)$ be an adjoint pair with unit $\eta$ and counit $\delta$.
  The composed endofunctor $\WW := \GG \circ \FF: \A \to \A$ together with the unit of the adjunction $\eta: \Id_\A \to \GG \circ \FF$ as the monad unit $\eta: \Id_\A \to \WW$ and 
  \[
    \mu:=\GG\delta \FF:\WW^2 \to \WW
  \]
  as the monad multiplication is a monad.
  It is called the \textbf{monad associated to the adjoint pair}.
\end{rmrk}

\subsection{Idempotent monads and reflective localizations} \label{subsect:reflective}

\begin{defn}
  A monad $(\WW,\eta,\mu)$ is called an \textbf{idempotent monad} if the multiplication $\mu: \WW^2 \to \WW$ is a natural isomorphism.
\end{defn}

\begin{defn}
  A full subcategory $\B \subset \A$ is called a \textbf{reflective subcategory} if the inclusion functor $\iota: \B \to \A$ has a left adjoint, called the \textbf{reflector}.
\end{defn}

\begin{prop}[{\cite[§4.2]{Borceux_Handbook2}}] \label{prop:reflective}
  Let $(\WW,\widetilde{\eta},\mu)$ be a monad on the category $\A$.
  Then the following statements are equivalent:
  \begin{enumerate}
    \item The monad $(\WW,\widetilde{\eta},\mu)$ is idempotent.
    \item The two natural transformations $\WW \widetilde{\eta}, \widetilde{\eta} \WW:\WW \to \WW^2$ are \emph{equal}.
    \item The (essential) image of $\WW$ in $\A$ is a reflective subcategory.
    \item The corestriction $\widehat{\GG}:=\cores_{\WW(\A)} \WW$ of $\WW$ to its (essential) image $\WW(\A)$ is left adjoint to the inclusion functor $\iota$ of $\WW(\A)$ into $\A$: $\widehat{\GG} \dashv (\iota: \WW(\A) \to \A)$.
    \item There exists an adjunction $\FF \dashv (\GG: \B \to \A)$ with unit $\eta$ and counit $\delta$ where $\GG$ is fully faithful such that its associated monad $(\GG \circ \FF, \eta, \GG \delta \FF)$ is isomorphic to $(\WW,\widetilde{\eta},\mu)$.
  \end{enumerate}
\end{prop}

A localization of a category $\A$ at collection of morphisms $\Sigma$ is roughly speaking the process of adjoining formal inverses of the morphisms in $\Sigma$ (cf.~\cite[Definition~10.3.1]{weihom}).

\begin{defn}[{\cite[§I.1.1]{GabZis}}] \label{defn:idem}
  A localization $\A \to \A[\Sigma^{-1}]$ (at a subset of morphisms $\Sigma$ of $\A$) is called \textbf{reflective localization} if it admits a fully faithful right adjoint $\GG:\A[\Sigma^{-1}] \rightarrow \A$.
\end{defn}

The following proposition states, among other things, that all functors having fully faithful right adjoints are in fact reflective localizations.

\begin{prop}[{\cite[Proposition~I.1.3]{GabZis}}] \label{prop:idempotent}
  Let $\FF \dashv (\GG: \B \to \A)$ be a pair of adjoint functors with unit $\eta$ and counit $\delta$.
  Further let $(\WW,\eta,\mu) := (\GG \circ \FF, \eta, \GG \delta \FF)$ be the associated monad.
  The following statements are equivalent:
  \begin{enumerate}
    \item $\GG$ is fully faithful (inducing an equivalence between $\B$ and its (essential) image, a reflective subcategory of $\A$).
    \item The associated monad $(\WW,\eta,\mu)$ is idempotent\footnote{In this case one says that the adjunction $\FF \dashv (\GG: \B \to \A)$ is idempotent.}, $\FF$ is essentially surjective, and $\GG$ is conservative\footnote{i.e., $\GG$ reflects isomorphisms: $\GG(\alpha)$ isomorphism $\implies$ $\alpha$ isomorphism.}.
    \item The counit $\delta:\FF \circ \GG \to Id_\B$ is a natural isomorphism\footnote{In particular, $\GG$ is a right inverse of $\FF$ and $\FF$ is essentially surjective.}.\label{prop:idempotent:delta}
    \item If $\Sigma$ is the set of morphisms $\phi \in \A$ such that $\FF(\phi)$ is invertible in $\B$, then $\FF:\A \to \B$ realizes the reflective localization $\A \to \A[\Sigma^{-1}]$.
  \end{enumerate}
  In particular, $\WW(\A) \simeq \GG(\B) \simeq \B \simeq \A[\Sigma^{-1}]$.
\end{prop}

\section{\nameft{Gabriel} localizations} \label{sect:Gabriel}

In this section we recall some results about \nameft{Serre} quotients arising from reflective localizations. From now on $\A$ is an \nameft{Abel}ian category.

\begin{defn}[{\cite[Exer.~10.3.2]{weihom}}]
  A non-empty full subcategory $\mathcal{C}$ of an \nameft{Abel}ian category $\A$ is called \textbf{thick}\footnote{\nameft{Gabriel} uses the notion \emph{\'epaisse}.
  Thick subcategories are automatically replete.} if it is closed under passing to subobjects, factor objects, and extensions. In other words, for every short exact sequence
  \[
    0 \to M' \to M \to M'' \to 0
  \]
  in $\A$ the object $M$ lies in $\C$ if and only if $M'$ and $M''$ lie in $\C$.
\end{defn}

\nameft{Pierre Gabriel} defines in his thesis \cite[§III.1]{Gab_thesis} the \textbf{(\nameft{Serre}) quotient category} $\A/\C$ with the same objects as $\A$ and with $\Hom$-groups defined by
\[
  \Hom_{\A/\C}(M,N) = \varinjlim_{M',N'} \Hom_\A(M',N/N')\mbox{,}
\]
where the direct limit is taken over all subobjects $M' \leq M$ and $N' \leq N$ such that $M/M'$ and $N'$ belong to $\C$.
Further he defines the \textbf{canonical functor} $\QQ:\A \to \A/\C$ which is the identity on objects\footnote{So $\QQ$ is surjective (on objects) but in general neither faithful nor full.} and which maps a morphism $\phi \in \Hom_\A(M,N)$ to its image in $\Hom_{\A/\C}(M,N)$ under the maps
\[
  \Hom(M' \hookrightarrow M, N \twoheadrightarrow N/N'):\Hom_\A(M,N) \to \Hom_\A(M',N/N') \mbox{.}
\]

\nameft{Gabriel} proves in \cite[Proposition~III.1.1]{Gab_thesis} that $\A/\C$ is an \nameft{Abel}ian category and that the canonical functor $\QQ: \A \to \A/\C$ is exact.
The canonical functor $\QQ: \A \to \A/\C$ fulfills the following universal property.
\begin{prop}[{\cite[Cor.~III.1.2]{Gab_thesis}}] \label{prop:universal_property_of_Q}
  Let $\C$ be a thick subcategory of $\A$ and $\GG:\A \to \mathcal{D}$ an exact functor into the \nameft{Abel}ian category $\mathcal{D}$.
  If $\GG(\C)$ is zero then there exists a unique functor $\HH:\A/\C \to \mathcal{D}$ such that $\GG=\HH \circ \QQ$.
\end{prop}

\begin{defn} \label{defn:Gabriel_monad}
  The thick subcategory $\C$ of the \nameft{Abel}ian category $\A$ is called a \textbf{localizing subcategory} if the canonical functor $\QQ: \A \to \A/\C$ admits a right adjoint $\SS:\A/\C \to \A$, i.e., a functor together with a binatural isomorphism
  \[
    \Hom_{\A/\C}(\QQ(-),-) \simeq \Hom_\A(-,\SS(-))\mbox{.}
  \]
  $\SS$ is called the\footnote{Cf.~Remark~\ref{rmrk:right_adjoints}.} \textbf{section functor} of $\QQ$.
  Denote by
  \[
    \delta: \QQ \circ \SS \to \Id_{\A/\C} \quad \mbox{and} \quad \eta:\Id_\A \to \SS \circ \QQ
  \]
  the counit and unit of the adjunction $\SS \dashv \QQ$, respectively.
  
  We call such a canonical functor $\QQ$ a \textbf{\nameft{Gabriel} localization} and the associated monad $(\SS \circ \QQ,\eta,\mu=\SS\delta \QQ)$ the \textbf{\nameft{Gabriel} monad}.
\end{defn}

The following definition describes the $\Sigma$-local objects in the sense of \nameft{Gabriel-Zisman}, where $\Sigma$ is the collection of all morphisms in $\A$ with both kernel and cokernel in $\C$.
\begin{defn}\label{defn:sat}
  Let $\C$ be a thick subcategory of $\A$.
  We call an object $M$ in $\A$ \textbf{$\C$-saturated} (or \textbf{$\C$-closed\footnote{\nameft{Gabriel} uses ferm\'e.}}) if it has no non-trivial subobjects in $\C$ and every extension of $M$ by an object $C \in \C$ is trivial, i.e., each short exact sequence $0 \to M \to E \to C \to 0$ in $\A$ is split.
  Define
  \[
    \Sat_\C(\A) := \mbox{full (replete) subcategory of $\C$-saturated objects}.
  \]
\end{defn}

  The following proposition together with Proposition~\ref{prop:idempotent} imply that $\A/\C$ is a reflective localization; more precisely $\A/\C \simeq \A[\Sigma^{-1}] \simeq \Sat_\C(\A)$, where $\Sigma$ is the collection of all morphisms in $\A$ with both kernel and cokernel in $\C$.

\begin{prop}[{\cite[Proposition~III.2.3 and its corollary]{Gab_thesis}}]\label{prop:localization}
  Let $\C \subset \A$ be a localizing subcategory of the \nameft{Abel}ian category $\A$ with section functor $\SS:\A/\C \to \A$.
  \begin{enumerate}
    \item $\SS:\A/\C \to \A$ is left exact and preserves products\footnote{Right adjoint functors preserve kernels and products~\cite[Theorem~II.7.7]{HS}.}. \label{prop:localization.S}
    \item The counit of the adjunction \label{prop:localization.delta} $\delta: \QQ \circ \SS \xrightarrow{\sim} \Id_{\A/\C}$ is a natural isomorphism.
    \item An object $M$ in $\A$ is $\C$-saturated if and only if $\eta_M:M \to (\SS \circ \QQ)(M)$ is an isomorphism, where $\eta$ is the unit of the adjunction. \label{prop:localization.eta}
  \end{enumerate}
\end{prop}

\begin{coro} \label{coro:localization}
  Let $\C \subset \A$ be a localizing subcategory of the \nameft{Abel}ian category $\A$.
  Then $\SS(\A/\C)\simeq \Sat_\C(\A)$ are reflective subcategories of $\A$, in particular, the following holds:
  \begin{enumerate}
    \item The section functor $\SS$ is fully faithful.
    \item $\SS$ is conservative. \label{coro:localization.conservative}
    \item $\eta (\SS \circ \QQ) = (\SS \circ \QQ) \eta$. \label{coro:localization.commute}
  \end{enumerate}
\end{coro}
\begin{proof}
  Proposition~\ref{prop:localization}.\eqref{prop:localization.delta} establishes the context of Propositions~\ref{prop:idempotent} and \ref{prop:reflective}.
\end{proof}

The existence of enough $\C$-saturated objects turns $\Sat_\C(\A)$ into an ``orthogonal complement'' of $\C$.
\begin{defn}\label{defn:enough_sat}
  Let $\C$ be a thick subcategory of $\A$.
  We say that $\A$ \textbf{has enough $\C$-saturated objects} if for each $M \in \A$ there exists a $\C$-saturated object $N$ and a morphism $\eta_M: M \to N$ such that $H_\C(M) := \ker \eta_M \in \C$.
\end{defn}
It follows from the fundamental theorem on homomorphisms that $H_\C(M)$\footnote{The notation $H_\C$ is motivated by the notation for local cohomology.} is the maximal subobject of $M$ in $\C$.

\begin{prop}[{\cite[Proposition~III.2.4]{Gab_thesis}}] \label{prop:localizing}
   The thick subcategory $\C \subset \A$ is loca\-lizing iff $\A$ has enough $\C$-saturated objects.
\end{prop}

\begin{rmrk} \label{rmrk:Hom_avoiding_limit}
The existence of the maximal subobject in $\C$ simplifies the description of $\Hom$-groups in $\A/\C$, namely
\[
  \Hom_{\A/\C}(M,N) = \varinjlim_{\substack{M' \hookrightarrow M \\ M / M' \in \C}} \Hom_\A(M',N/H_\C(N)) \mbox{.}
\]
The existence of the section functor $\SS$ as the right adjoint of $\QQ$ reduces the computability of $\Hom$-groups in $\A/\C \simeq \Sat_\C(\A)$ even further to their computability in $\A$
\[
  \Hom_{\A/\C}(M,N) = \Hom_{\A/\C}(\QQ(M),\QQ(N)) \cong \Hom_\A(M, (\SS\circ \QQ)(N)) \mbox{,}
\]
avoiding the direct limit in the definition of $\Hom_{\A/\C}$ completely.
\end{rmrk}

\begin{coro}\label{coro:imS}
  The image $\SS(\A/\C)$ of $\SS$ is a subcategory of $\Sat_\C(\A)$ and the inclusion functor $\SS(\A/\C)\hookrightarrow \Sat_\C(\A)$ is an equivalence of categories with the restricted-corestricted \nameft{Gabriel} monad $\SS\circ \QQ:\Sat_\C(\A) \to \SS(\A/\C)$ as a quasi-inverse.
  In other words, $\Sat_\C(\A)$ is the essential image of $\SS$ and, hence, of the \nameft{Gabriel} monad $\SS\circ \QQ$.
\end{coro}
\begin{proof}
  The inclusion of $\SS(\QQ(\A))=\SS(\A/\C)$ in $\Sat_\C(\A)$ will follow from Proposition~\ref{prop:localization}.\eqref{prop:localization.eta} as soon as we can show that $\eta \SS$, or equivalently $\eta(\SS\circ \QQ)$, is a natural isomorphism.
  Proposition~\ref{prop:idempotent} applied to Proposition~\ref{prop:localization}.\eqref{prop:localization.delta} states that the multiplication $\mu$ of the \nameft{Gabriel} monad $\SS\circ \QQ$ is a natural isomorphism.
  Finally, the second coherence condition $\mu \circ \eta (\SS \circ \QQ) =\Id_{\SS \circ \QQ}$ implies that $\eta (\SS \circ \QQ)=\mu^{-1}$ is also a natural isomorphism.
\end{proof}
\begin{coro} \label{coro:monad}
  The restricted canonical functor $\QQ:\Sat_\C(\A) \to \A/\C$ and the corestricted section functor $\SS:\A/\C \to \Sat_\C(\A)$ are quasi-inverse equivalences of categories.
  In particular, $\Sat_\C(\A) \simeq \SS(\A/\C) \simeq \A/\C$ is an \nameft{Abel}ian category.
\begin{center}
\begin{tikzpicture}[smooth, x=2.0cm, y=2.0cm]
   \node (Sat) at (0,0) {$\Sat_\C(\A)$};
   \node (S) at ($(Sat)+(2.75,0)$) {$\SS(\A/\C)=\SS(\QQ(\A))$};
   \node (AC) at ($2*(S)$) {$\A/\C$};
   \node (A) at ($(S)+(0,1.5)$) {$\A$};
   
   \draw[-doublestealth] ($(AC.west)-(0,0.05)$) -- node[below] {$\SS$} ($(S.east)-(0,0.05)$);
   \draw[stealth'-] ($(AC.west)+(0,0.05)$) -- node[above] {$\QQ$} ($(S.east)+(0,0.05)$);
   
   \draw[left hook-stealth'] ($(S.west)+(0,0.05)$) -- ($(Sat.east)+(0,0.05)$);
   \draw[stealth'-] ($(S.west)-(0,0.05)$) -- node[below] {$\SS\circ \QQ_{|\Sat_\C(\A)}$} ($(Sat.east)-(0,0.05)$);
   
   \draw[doublestealth-] ($(AC.north)+(0.06,0)$) to [bend right=25] node[above right] {$\QQ$} ($(A.east)+(0,0.04)$);
   \draw[-stealth'] ($(AC.north)-(0.07,0)$) to [bend right=25] node[below left] {$\SS$} ($(A.east)-(0,0.04)$);
   
   \draw[right hook-stealth'] ($(S.north)-(0.05,0)$) -- ($(A.south)-(0.05,0)$);
   \draw[doublestealth-] ($(S.north)+(0.05,0)$) to node[right] {$\SS \circ \QQ$} ($(A.south)+(0.05,0)$);
   
   \draw[left hook-stealth'] ($(Sat.north)+(0.07,0)$) to [bend left=25] node[below right] {$\iota$} ($(A.west)-(0,0.04)$);
   \draw[stealth'-]
   [
     postaction={
       decoration={
         text along path,raise=5pt,text align={left indent=5pt},
         text={${\widehat{\QQ}:=\operatorname{cores}_{{\Sat}_{\C}({\A})}}(\SS{\circ}\QQ)\ $}
       },decorate
     }
   ] ($(Sat.north)-(0.06,0)$) to [bend left=25] ($(A.west)+(0,0.04)$);
\end{tikzpicture}
\end{center}
  Define the reflector $\widehat{\QQ} := \cores_{\Sat_\C(\A)} (\SS \circ \QQ): \A \to \Sat_\C(\A)$.
  The adjunction $\widehat{\QQ} \dashv (\iota: \Sat_\C(\A) \hookrightarrow \A)$ corresponds under the above equivalence to the adjunction $\QQ \dashv (\SS: \A/\C \to \A)$. They both share the same adjunction monad $\SS \circ \QQ = \iota \circ \widehat{\QQ}: \A \to \A$.
  In particular, the reflector $\widehat{\QQ}$ is exact\footnote{$\Sat_\C(\A)$ is then a \nameft{Giraud} subcategory.} and $\iota$ is left exact.
\end{coro}

Let $\C \subset \A$ be a localizing subcategory of the \nameft{Abel}ian category $\A$.
$\SS(\A/\C) \simeq \Sat_\C(\A)$ are not in general \nameft{Abel}ian \emph{sub}categories of $\A$ in the sense of \cite[p.~7]{weihom}, as short exact sequences in $\Sat_\C(\A)$ are not necessarily exact in $\A$.
This is due to the fact that the left exact section functor $\SS$ (cf.~Proposition~\ref{prop:localization}.\eqref{prop:localization.S}) is in general not exact;
the cokernel in $\Sat_\C(\A) \subset \A$ differs from the cokernel in $\A$ (cf.~\cite{BL_GabrielMorphisms}).

The full subcategory $\A_\C \subset \A$ consisting of all objects with no nontrivial subobject in $\C$ is a pre-\nameft{Abel}ian category\footnote{although not an \nameft{Abel}ian category, in general, as monomorphisms need not be kernels of their cokernels. For example, the monomorphism $2:\Z \to \Z$ in $\A_\C = \{ \mbox{f.g.~torsion-free \nameft{Abel}ian groups} \}$ is not kernel of its cokernel (which is zero), where $\A := \{ \mbox{f.g.~\nameft{Abel}ian groups} \} \supset \C := \{ \mbox{f.g.~torsion \nameft{Abel}ian groups} \}$. This gives an example of a morphism which is monic and epic in $\A_\C$ but not an isomorphism.}.
The kernel of a morphisms in $\A_\C$ is its kernel in $\A$.
The cokernel of a morphism in $\A_\C$ is isomorphic to its cokernel as a morphism in $\A$ modulo the maximal subobject in $\C$, in case $\C$ is a localizing subcategory (cf.~Proposition~\ref{prop:localizing}).

An in that case $\Sat_\C(\A) \subset \A_\C$ is the \textbf{completion} of $\A_\C$ with respect to the property that every extension by an object in $\C$ is trivial.
This completion is given by the \nameft{Gabriel} monad $\SS \circ \QQ$ restricted to $\A_\C$.

\section{Characterizing \nameft{Gabriel} localizations and \nameft{Gabriel} monads} \label{sect:Characterization}

The next proposition states that in fact all exact reflective localizations in the setup \nameft{Abel}ian categories are (reflective) \nameft{Gabriel} localizations.

\begin{prop}[{\cite[Proposition~III.2.5]{Gab_thesis}, \cite[Chap.~1.2.5.d]{GabZis}}]\label{prop:local_ker}
  Let $\widetilde{\QQ} \dashv (\widetilde{\SS}:\B \to \A)$ be a pair of adjoint functors of \nameft{Abel}ian categories.
  Assume, that $\widetilde{\QQ}$ is exact and the counit $\delta:\widetilde{\QQ}\circ{\widetilde{\SS}} \to \Id_{\B}$ of the adjunction is a natural isomorphism.
  Then $\C := \ker \widetilde{\QQ}$ is a localizing subcategory of $\A$ and the adjunction $\widetilde{\QQ} \dashv (\widetilde{\SS}:\B \to \A)$ induces an adjoint equivalence from $\B$ to $\A/\C$.
\end{prop}

Now, we approach the central definition of this paper which collects some properties of \nameft{Gabriel} monads.

\begin{defn}\label{defn:Csaturating}
  Let $\C \subset \A$ be a localizing subcategory of the \nameft{Abel}ian category $\A$ and $\iota: \Sat_\C(\A) \hookrightarrow \A$ the full embedding.
  We call an endofunctor $\WW:\A \to \A$ together with a natural transformation $\widetilde{\eta}:\Id_\A \to \WW$ \textbf{$\C$-saturating} if the following holds:
  \begin{enumerate}
    \item $\C \subset \ker \WW$, \label{defn:Csaturating:ker}
    \item $\WW(\A) \subset \Sat_\C(\A)$, \label{defn:Csaturating:im}
    \item $\GG := \cores_{\Sat_\C(\A)} \WW$ is exact, \label{defn:Csaturating:exact}
    \item $\widetilde{\eta} \WW = \WW \widetilde{\eta}$, and \label{defn:Csaturating:idem}
    \item $\widetilde{\eta} \iota: \Id_{\A \mid \Sat_\C(\A)} \to \WW_{|\Sat_\C(\A)}$ is a natural isomorphism\footnote{In particular, $\WW(\A)$ is an essentially wide subcategory of $\Sat_\C( \A)$.}. \label{defn:Csaturating:eta}
  \end{enumerate}
  Let $\HH$ be the unique functor from Proposition~\ref{prop:universal_property_of_Q} such that $\GG=\HH \circ \QQ$.
  We call the composed functor $\widetilde{\HH} := \iota \circ \HH$ the \textbf{colift of $\WW$ along $\QQ$}, since $\WW = \iota \circ \GG = \widetilde{\HH} \circ \QQ$.
\end{defn}

\begin{lemm} \label{lemm:GabrielSaturating}
  Let $\QQ: \A \to \A/\C$ be a \nameft{Gabriel} localization with section functor $\SS$.
  Then each $\C$-saturating endofunctor $\WW$ of $\A$ is naturally isomorphic to $\SS \circ \QQ$.
  Furthermore, the colift $\widetilde{\HH}$ of $\WW$ along $\QQ$ is also a section functor naturally isomorphic to $\SS$.
\end{lemm}
\begin{proof}
For $\GG := \cores_{\Sat_\C(\A)} \WW$ let $\HH$ be the unique functor from Proposition~\ref{prop:universal_property_of_Q} (using Definition~\ref{defn:Csaturating}.\eqref{defn:Csaturating:ker}) such that $\GG=\HH \circ \QQ$.
  \begin{align*}
    \GG
    &=
    \HH \circ \QQ \\
    &\simeq
    \HH\circ \QQ\circ \SS\circ \QQ & \mbox{($\Id_{\A/\C}\simeq\QQ\circ \SS$ by Proposition~\ref{prop:localization}.\eqref{prop:localization.delta})} \\
    &=
    \GG \circ \SS \circ \QQ \\
    &\simeq
    \Id_{\Sat_\C(\A)} \circ \SS\circ \QQ & \mbox{(using $\SS(\QQ(\A)) \subset \Sat_\C(\A)$ and \ref{defn:Csaturating}.\eqref{defn:Csaturating:eta})} \\
    &=
    \cores_{\Sat_\C(\A)} (\SS \circ \QQ) \mbox{.}
  \end{align*}
  Then, using the notation of Definition~\ref{defn:Csaturating},
  \[
    \WW = \widetilde{\HH} \circ \QQ = \iota \circ \GG \simeq \SS \circ \QQ \mbox{.}
  \]
  This also proves the equivalence $\widetilde{\HH} \simeq \SS$, as $\QQ$ is surjective.
\end{proof}

\begin{prop}\label{prop:C_saturating_monad}
  Let $\QQ: \A \to \A/\C$ be a \nameft{Gabriel} localization and $(\WW,\widetilde{\eta})$ be a \linebreak[3] $\C$-saturating endofunctor of $\A$ with colift $\widetilde{\HH}$ along $\QQ$.
  Then there exists a natural transformation $\widetilde{\delta}: \QQ \circ \widetilde{\HH} \to \Id_{\A/\C}$ such that $\QQ$ and $\widetilde{\HH}$ form an adjoint pair $\QQ \dashv \widetilde{\HH}$ with unit $\widetilde{\eta}$ and counit $\widetilde{\delta}$.
\end{prop}
\begin{defn}
Hence, each $\C$-saturating endofunctor $(\WW,\widetilde{\eta})$ is the monad $(\WW,\widetilde{\eta},\widetilde{\HH}\widetilde{\delta}\QQ)$ associated to the adjunction $\QQ \dashv \widetilde{\HH}$.
We call it a \textbf{$\C$-saturating monad}.
\end{defn}
\begin{proof}[Proof of Proposition~\ref{prop:C_saturating_monad}]
  We define a natural transformation $\widetilde{\delta}: \QQ \circ \widetilde{\HH} \to \Id_{\A/\C}$ and check the two zig-zag identities \eqref{zigzag}, i.e., that the compositions of natural transformations
  \[
      \QQ \xrightarrow{\QQ\widetilde{\eta}} \QQ\circ \widetilde{\HH} \circ \QQ \xrightarrow{\widetilde{\delta} \QQ} \QQ
      \quad \mbox{and} \quad
      \widetilde{\HH} \xrightarrow{\widetilde{\eta} \widetilde{\HH}} \widetilde{\HH}\circ \QQ \circ  \widetilde{\HH} \xrightarrow{\widetilde{\HH} \widetilde{\delta}} \widetilde{\HH}
  \]
  are the identity of functors.
  By \ref{defn:Csaturating}.\eqref{defn:Csaturating:eta} we know that $(\widetilde{\HH} \circ \QQ) \widetilde{\eta} = \WW \widetilde{\eta} \stackrel{\ref{defn:Csaturating}.\text{\eqref{defn:Csaturating:idem}}}{=} \widetilde{\eta} \WW = (\widetilde{\eta}\iota) \GG$ is an isomorphism.
  Hence, also $\QQ \widetilde{\eta}$ is a natural isomorphism, because the functor $\widetilde{\HH}$ is equivalent to $\SS$ by Lemma~\ref{lemm:GabrielSaturating} and, thus, reflects isomorphisms by Lemma~\ref{coro:localization}.\eqref{coro:localization.conservative}.  
  This allows us to define $\widetilde{\delta}$ in such a way to satisfy the first zig-zag identity, i.e., set $\widetilde{\delta} \QQ := (\QQ \widetilde{\eta})^{-1}$.
  This defines $\widetilde{\delta}$ as $\QQ$ is surjective (on objects).
  The second zig-zag identity is equivalent, again due to the surjectivity of $\QQ$, to the second zig-zag identity applied to $\QQ$, i.e., $(\widetilde{\HH} \widetilde{\delta} \QQ) \circ \widetilde{\eta} (\widetilde{\HH} \circ \QQ)$ being the identity transformation of the functor $\widetilde{\HH} \circ \QQ=\WW$.
  Now
  \begin{align*}
    (\widetilde{\HH} \widetilde{\delta} \QQ) \circ \widetilde{\eta} (\widetilde{\HH} \circ \QQ)
    &=
    (\widetilde{\HH} (\QQ \widetilde{\eta})^{-1}) \circ \widetilde{\eta} (\widetilde{\HH} \circ \QQ) & \mbox{(by the definition $\widetilde{\delta} \QQ := (\QQ \widetilde{\eta})^{-1}$)} \\
    &=
    ((\widetilde{\HH} \circ \QQ) \widetilde{\eta})^{-1} \circ \widetilde{\eta} (\widetilde{\HH} \circ \QQ) \\
    &=
    (\widetilde{\eta} (\widetilde{\HH} \circ \QQ))^{-1} \circ \widetilde{\eta} (\widetilde{\HH} \circ \QQ) & \mbox{(using $\widetilde{\eta} (\widetilde{\HH} \circ \QQ) \stackrel{\text{\ref{defn:Csaturating}.\eqref{defn:Csaturating:idem}}}{=} (\widetilde{\HH} \circ \QQ) \widetilde{\eta}$)} \\
    &=
    \Id_{\widetilde{\HH} \circ \QQ} \mbox{.}
  \end{align*}
\end{proof}

We now approach our main result.
\begin{theorem}[Characterization of \nameft{Gabriel} monads] \label{thm:equivalent_section}
  Each \nameft{Gabriel} monad is a $\C$-saturating monad.
  Conversely, each $\C$-saturating monad is equivalent to a \nameft{Gabriel} monad.
\end{theorem}
\begin{proof}
The conditions in Definition~\ref{defn:Csaturating} clearly apply to a \nameft{Gabriel} monad $\SS\circ \QQ$ by definition of the canonical functor $\QQ$, Corollary~\ref{coro:imS}, Corollary~\ref{coro:monad}, Corollary~\ref{coro:localization}.\eqref{coro:localization.commute}, and Proposition~\ref{prop:localization}.\eqref{prop:localization.eta}, respectively.

The converse follows directly from Lemma~\ref{lemm:GabrielSaturating} and Proposition~\ref{prop:C_saturating_monad} which prove that the two adjunctions $\QQ \dashv \widetilde{\HH}$ and $\QQ \dashv \SS$ are equivalent, and so are their associated monads.
\end{proof}

\def\cprime{$'$} \def\cprime{$'$} \def\cprime{$'$} \def\cprime{$'$}
  \def\cprime{$'$}
\providecommand{\bysame}{\leavevmode\hbox to3em{\hrulefill}\thinspace}
\providecommand{\MR}{\relax\ifhmode\unskip\space\fi MR }
\providecommand{\MRhref}[2]{%
  \href{http://www.ams.org/mathscinet-getitem?mr=#1}{#2}
}
\providecommand{\href}[2]{#2}


\begin{thebibliography}{{Coq}04}

\bibitem[BGG78]{BGG}
I.~N. Bern{\v{s}}te{\u\i}n, I.~M. Gel{\cprime}fand, and S.~I. Gel{\cprime}fand,
  \emph{Algebraic vector bundles on {${\bf P}\sp{n}$} and problems of linear
  algebra}, Funktsional. Anal. i Prilozhen. \textbf{12} (1978), no.~3, 66--67.
  \MR{MR509387 (80c:14010a)}

\bibitem[BLH11]{BL}
Mohamed Barakat and Markus Lange-Hegermann, \emph{An axiomatic setup for
  algorithmic homological algebra and an alternative approach to localization},
  J. Algebra Appl. \textbf{10} (2011), no.~2, 269--293,
  (\href{http://arxiv.org/abs/1003.1943}{\texttt{arXiv:1003.1943}}).
  \MR{2795737 (2012f:18022)}

\bibitem[BLH14a]{BL_SerreQuotients}
Mohamed Barakat and Markus Lange-Hegermann, \emph{Characterizing {S}erre
  quotients with no section functor and applications to coherent sheaves},
  Appl. Categ. Structures \textbf{22} (2014), no.~3, 457--466,
  (\href{http://arxiv.org/abs/1210.1425}{\texttt{arXiv:1210.1425}}).
  \MR{3200455}

\bibitem[BLH14b]{BL_Sheaves}
Mohamed Barakat and Markus Lange-Hegermann, \emph{A constructive approach to
  coherent sheaves via {G}abriel monads},
  (\href{http://arxiv.org/abs/1409.6100}{\texttt{arXiv:1409.6100}}), 2014.

\bibitem[BLH14c]{BL_GabrielMorphisms}
Mohamed Barakat and Markus Lange-Hegermann, \emph{{G}abriel morphisms and the
  computability of {S}erre quotients with applications to coherent sheaves},
  (\href{http://arxiv.org/abs/1409.2028}{\texttt{arXiv:1409.2028}}), 2014.

\bibitem[BLH14d]{BL_ExtComputability}
Mohamed Barakat and Markus Lange-Hegermann, \emph{On the
  $\mathrm{Ext}$-computability of {S}erre quotient categories}, Journal of
  Algebra \textbf{420} (2014), no.~0, 333--349,
  (\href{http://arxiv.org/abs/1212.4068}{\texttt{arXiv:1212.4068}}).

\bibitem[Bor94a]{Borceux_Handbook1}
Francis Borceux, \emph{Handbook of categorical algebra. 1}, Encyclopedia of
  Mathematics and its Applications, vol.~50, Cambridge University Press,
  Cambridge, 1994, Basic category theory. \MR{1291599 (96g:18001a)}

\bibitem[Bor94b]{Borceux_Handbook2}
Francis Borceux, \emph{Handbook of categorical algebra. 2}, Encyclopedia of
  Mathematics and its Applications, vol.~51, Cambridge University Press,
  Cambridge, 1994, Categories and structures. \MR{1313497 (96g:18001b)}

\bibitem[BS98]{BrSh}
M.~P. Brodmann and R.~Y. Sharp, \emph{Local cohomology: an algebraic
  introduction with geometric applications}, Cambridge Studies in Advanced
  Mathematics, vol.~60, Cambridge University Press, Cambridge, 1998.
  \MR{1613627 (99h:13020)}

\bibitem[{Coq}04]{coq:manual}
{Coq development team}, \emph{The coq proof assistant reference manual},
  LogiCal Project, 2004, Version 8.0.

\bibitem[EFS03]{EFS}
David Eisenbud, Gunnar Fl{\o}ystad, and Frank-Olaf Schreyer, \emph{Sheaf
  cohomology and free resolutions over exterior algebras}, Trans. Amer. Math.
  Soc. \textbf{355} (2003), no.~11, 4397--4426 (electronic). \MR{MR1990756
  (2004f:14031)}

\bibitem[Gab62]{Gab_thesis}
Pierre Gabriel, \emph{Des cat\'egories ab\'eliennes}, Bull. Soc. Math. France
  \textbf{90} (1962), 323--448. \MR{0232821 (38 \#1144)}

\bibitem[Gro57]{Tohoku}
Alexander Grothendieck, \emph{Sur quelques points d'alg\`ebre homologique},
  T\^ohoku Math. J. (2) \textbf{9} (1957), 119--221. \MR{MR0102537 (21 \#1328)}

\bibitem[GZ67]{GabZis}
P.~Gabriel and M.~Zisman, \emph{Calculus of fractions and homotopy theory},
  Ergebnisse der Mathematik und ihrer Grenzgebiete, Band 35, Springer-Verlag
  New York, Inc., New York, 1967. \MR{0210125 (35 \#1019)}

\bibitem[HS97]{HS}
P.~J. Hilton and U.~Stammbach, \emph{A course in homological algebra}, second
  ed., Graduate Texts in Mathematics, vol.~4, Springer-Verlag, New York, 1997.
  \MR{MR1438546 (97k:18001)}

\bibitem[{nLa}12]{nLab}
{nLab authors}, \emph{The $n$lab}, 2012, (\url{http://ncatlab.org/nlab/})
  [Online; accessed 19-August-2014].

\bibitem[Ser55]{FAC}
Jean-Pierre Serre, \emph{Faisceaux alg\'ebriques coh\'erents}, Ann. of Math.
  (2) \textbf{61} (1955), 197--278. \MR{MR0068874 (16,953c)}

\bibitem[Sim06]{SimpsonGZComputer}
Carlos Simpson, \emph{Explaining {G}abriel-{Z}isman localization to the
  computer}, J. Automat. Reason. \textbf{36} (2006), no.~3, 259--285.
  \MR{2288803 (2007h:68176)}

\bibitem[Tho11]{ThomasArrow}
Sebastian Thomas, \emph{On the 3-arrow calculus for homotopy categories},
  Homology Homotopy Appl. \textbf{13} (2011), no.~1, 89--119,
  (\href{http://arxiv.org/abs/1001.4536}{\texttt{arXiv:1001.4536}}).
  \MR{2803869}

\bibitem[Ver96]{Verdier}
Jean-Louis Verdier, \emph{Des cat\'egories d\'eriv\'ees des cat\'egories
  ab\'eliennes}, Ast\'erisque (1996), no.~239, xii+253 pp. (1997), With a
  preface by Luc Illusie, Edited and with a note by Georges Maltsiniotis.
  \MR{1453167 (98c:18007)}

\bibitem[Wei94]{weihom}
Charles~A. Weibel, \emph{An introduction to homological algebra}, Cambridge
  Studies in Advanced Mathematics, vol.~38, Cambridge University Press,
  Cambridge, 1994. \MR{MR1269324 (95f:18001)}

\end{thebibliography}
\end{document}